\documentclass[a4paper,12pt]{article}
\usepackage{amsmath,amssymb,amsfonts,amscd,amsthm}
\usepackage{graphicx}
\usepackage{dsfont}
\usepackage{times}
\usepackage{blindtext}
\usepackage[titletoc, toc, title, page]{appendix}
\usepackage{color}
\usepackage{enumerate}
\usepackage{fancyhdr}		
\usepackage{pgf,tikz}
\usepackage{sectsty} 
\definecolor{bleu1}{RGB}{0,57,128}
\definecolor{emerald}{rgb}{0.31, 0.78, 0.47}
\definecolor{deepjunglegreen}{rgb}{0.0, 0.29, 0.29}
\usepackage[utf8]{inputenc}
\usepackage[english]{babel}

\newtheorem{Main}{Theorem}

\usepackage[colorlinks,linkcolor=deepjunglegreen,anchorcolor=green,citecolor=bleu1]{hyperref}

\usepackage{xcolor}
\definecolor{dgreen}{rgb}{0.1,0.6,0.1}

\definecolor{bluegray}{rgb}{0.4, 0.6, 0.8}

\usepackage{xcolor}

\newtheorem{theorem}{Theorem}                             
\newtheorem{lemma}{Lemma}
\newtheorem{proposition}{Proposition}

\numberwithin{equation}{section}

\makeatletter
\def\author#1{\gdef\autrun{\def\and{\unskip, }#1}\gdef\@author{#1}}
\def\address#1{{\def\and{\\\hspace*{18pt}}\renewcommand{\thefootnote}{}%
\footnote {#1}}%
}

\makeatother
\def\email#1{e-mail: #1}

\def\keywords#1{\par\medskip
\noindent\textbf{Keywords.} #1}

\renewenvironment{proof}{\noindent {\textit{Proof.}}}{\qed\vskip 0.2cm}

\makeatletter
\newenvironment{proofof}[1]{\par
  \pushQED{\qed}%
  \normalfont \topsep6\p@\@plus6\p@\relax
  \trivlist
  \item[\hskip\labelsep
    \textit{\textbf{Proof of} #1\@addpunct{.}}]\ignorespaces
}{%
  \popQED\endtrivlist\@endpefalse
}
\makeatother

\newcommand{\C}{\mathbb C}
\newcommand{\R}{\mathbb R}

\newcommand{\Z}{\mathbb Z}
\newcommand{\N}{\mathbb N}

\newcommand{\T}{\mathbb{T}}

\newcommand{\im}{\mathop{\mbox{Im}}}

\renewcommand{\min}{\mathop{\mbox{min}}}

\providecommand{\abs}[1]{\lvert#1\rvert}	
\providecommand{\norm}[1]{\lVert#1\rVert}

\begin{document}




\title{Positive measure of effective quasi-periodic motion near a 
Diophantine torus}

\author{Abed Bounemoura \and Gerard Farré}

\date{}

\maketitle

\address{A. Bounemoura: CNRS - PSL Research University; \email{abedbou@gmail.com}}

\address{G. Farré: KTH, Department of Mathematics; \email{gerardfp@kth.se}}

\abstract{
It was conjectured by Herman that an analytic Lagrangian Diophantine quasi-periodic torus $\mathcal{T}_0$, invariant by a real-analytic Hamiltonian system, is always accumulated by a set of positive Lebesgue measure of other Lagrangian Diophantine quasi-periodic invariant tori. While the conjecture is still open, we will prove the following weaker statement: there exists an open set of positive measure (in fact, the relative measure of the complement is exponentially small) around $\mathcal{T}_0$ such that the motion of all initial conditions in this set is ``effectively" quasi-periodic in the sense that they are close to being quasi-periodic for an interval of time which is doubly exponentially long with respect to the inverse of the distance to $\mathcal{T}_0$. This open set can be thought as a neighborhood of a hypothetical invariant set of Lagrangian Diophantine quasi-periodic tori, which may or may not exist.


\section{Introduction}
\label{sec_1}

 We will be considering the dynamics in a neighborhood of an analytic Lagrangian Diophantine quasi-periodic torus $\mathcal{T}_0$ (DQP torus for short) invariant by a real-analytic Hamiltonian system. It is well-known that without loss of generality, one may consider a real-analytic Hamiltonian $H:  \T^d \times U \rightarrow \R$, with $\T^d:=\R^d / \Z^d$ and $U\subset \R^d $ an open set containing the origin ($d \geq 2$ is an integer), of the form
\begin{equation}
\label{eq_main_ham}
 H(\theta,I)= \omega \cdot I + \mathcal{O}(\abs{I}^2)
\end{equation}
where $\omega \in \R^d$ satisfies a Diophantine condition with exponent $\tau \geq d-1$ and constant $\gamma > 0$ (in short, $\omega\in DC(\tau, \gamma))$: for all $k=(k_1,\dots,k_d)\in \Z^{d}\setminus \{0\}$,
\[ |\omega \cdot k| \geq \gamma |k|^{-\tau}, \quad |k|:=|k_1|+\cdots+|k_d|. \]  
Without loss of generality, we assume $|\omega|=1$. A fundamental object in the study of the dynamics near $\mathcal{T}_0= \T^d\times \{0\}$ is the so-called Birkhoff normal form (\cite{Bir66}, see also \cite{douady} for a more recent treatment). For any $m \in \N$, $m \geq 2$, there exists a polynomial $N_m(I)=\omega\cdot I+\mathcal{O}(\abs{I}^2)$ of degree $m$ and an analytic exact-symplectic transformation $\Psi_m(\theta, I) = (\theta + \mathcal{O}(|I|), I + \mathcal{O}(|I|^2))$ such that 
\begin{equation}\label{eq_after_change}
H\circ \Psi_m(\theta, I) = N_{m} (I) + \mathcal{O}(\abs{I}^{m+1}) .
\end{equation}
Moreover, these constructions have a formal limit as $m \rightarrow + \infty$; there exists a formal power series $N$ and a formal exact-symplectic transformation $\Psi$ such that formally
\[  H\circ \Psi(\theta, I) = N (I). \] 
The formal power series $N$ is unique and is called the Birkhoff normal form (BNF for short, the formal mapping $\Psi$ is also unique upon normalization) and it was proved recently by Krikorian that generically it is divergent (\cite{Kri19}, the corresponding statement for the transformation was proved much earlier by Siegel in~\cite{S54}). Yet the existence of the truncated BNF $N_m$ for any $m \geq 2$ has several consequences. 

Without further assumptions, given $r>0$ small enough, choosing an ``optimal" order of truncation $m=m(r)$ (Poincaré ``summation to the least term"), on a neighborhood of size $r$ around the origin, the remainder in \eqref{eq_after_change} $\mathcal{O}(\abs{I}^{m+1})$ can be made exponentially small with respect to $(1/r)^a$ where $a=1/(\tau+1)$, leading to the following estimates for all solutions $(\theta(t),I(t))$ with initial conditions $(\theta_0,I_0)$ with $|I_0|<r$: there exist positive constants $C$ and $c$ independent of $r$ and a vector $\omega(I_0) \in \R^d$ with ${|\omega(I_0)-\omega| \leq Cr}$ such that
\begin{equation}\label{effexp}
|\theta(t)-\theta_0-t\omega(I_0)| \leq Cr, \quad |I(t)-I_0|\leq Cr^2 \quad \text{for} \quad |t| \leq \exp(cr^{-a}).
\end{equation}
We refer to the general theorem contained in~\cite{jorba} and references therein for earlier results (the elementary fact that one can also control the evolution of the angles was pointed out in~\cite{Pop00}). Those estimates~\eqref{effexp} can be understood as an effective stability of the quasi-periodic motion for an exponentially long interval of time. In general, those estimates cannot be improved: it follows from a recent construction of an unstable DQP torus in~\cite{gerard} that the action variables cannot be stable for a longer interval of time, but stronger results can be obtained if one imposes non-degeneracy conditions on the BNF.

First, application of Nekhoroshev theory (\cite{Nek77}, \cite{Nek79}) yields stability of the action variables for a doubly exponentially long interval of time for all solutions with $|I_0|<r$ and $r$ small enough:
\begin{equation}\label{nek}
|I(t)-I_0|\leq Cr^2 \quad \text{for} \quad |t| \leq \exp\exp(cr^{-a}).
\end{equation}
This was first proved by Morbidelli and Giorgilli in~\cite{MG}) under a convexity assumption on $N$ (that is, the quadratic part of $N_2$ is positive definite) and later in \cite{BFN} under a generic ``steepness" assumption on $N$ (more precisely, it applies provided $N_{\bar m}$, where $\bar m$ depends only on $d$, avoids a semi-algebraic subset of positive codimension). Observe that in~\eqref{nek} only the time-scale of stability of the action variables is improved with respect to~\eqref{effexp}.

Applying classical KAM theory (\cite{kolmo,arnold,moser})), provided the quadratic part of $N_2$ is non-degenerate (this is Kolmogorov non-degeneracy), for $r>0$ small enough there exists a set of DQP tori $\mathcal{K}_r$ in $\T^d \times \{|I|<r\}$ with positive Lebesgue measure. In fact, the relative measure of the complement of $\mathcal{K}_r$ is bounded by $\exp(-cr^{-a})$. KAM theory extends to more general non-degeneracy assumptions (see~\cite{russ} for instance) and thus the statement on the existence of positive measure of DQP tori remains true provided that for some $m \geq 2$, the truncated BNF $N_m$ has the property that the local image of its gradient $\nabla N_m$ is not contained in any hyperplane of $\R^d$ (see~\cite{EFK} and also Theorem~\ref{cor_dimd} below for a more precise statement). Observe that when KAM theory applies, one can easily find an open set $\mathcal{U}_r$ in $\T^d \times \{|I|<r\}$ ``centered around" $\mathcal{K}_r$ such that for solutions with initial condition $(\theta_0,I_0) \in \mathcal{U}$, the effective stability of the quasi-periodic motion for an exponentially long interval of time given by~\eqref{effexp} is improved to a doubly exponentially long interval of time: 
\begin{equation}\label{effdoubleexp}
|\theta(t)-\theta_0-t\omega(I_0)| \leq Cr, \quad |I(t)-I_0|\leq Cr^2 \quad \text{for} \quad |t| \leq \exp\exp(cr^{-a}).
\end{equation}
In~\cite{herman_icm}, Herman conjectured that the conclusions of the application of KAM theory holds true without any non-degeneracy assumptions: more precisely, for $H$ as in~\eqref{eq_main_ham} with $\omega$ Diophantine, he asked whether there always exists, for $r>0$ small enough, a set of DQP tori $\mathcal{K}_r$ in $\T^d \times \{|I|<r\}$ with positive Lebesgue measure. In the case $d=2$ this follows from general results of R{\"u}ssmann (\cite{russmann}) and Bruno (\cite{Bru71}) (the arithmetic condition on $\omega$ can be slightly weakened) but for $d\geq 3$ this is an open problem; the only general result is due to Eliasson, Fayad and Krikorian (\cite{EFK}) who proved that $\mathcal{T}_0$ is always accumulated by other DQP tori but along an analytic submanifold, so with zero Lebesgue measure in general.    

The aim of this paper is to prove that the following formal consequence of Herman conjecture holds true; for $r>0$ small enough, there exists an open set $\mathcal{U}_r$ in $\T^d \times \{|I|<r\}$, the complement of which has a relative measure bounded by $ \exp(-cr^{-a})$, such that for solutions with initial condition $(\theta_0,I_0) \in \mathcal{U}_r$, one has effective stability of the quasi-periodic motion for a doubly exponentially long interval as expressed in~\eqref{effdoubleexp}. This is the content Theorem~\ref{thm_main} below. Of course, whether this set $\mathcal{U}_r$ actually contains a set of positive measure $\mathcal{K}_r$ of DQP tori is still an open question.

\section{Main results}
 
To state precisely our main results, let us introduce some notations. Given $\rho>0$ and $r>0$, we consider complex domains 
\[ \T^d_{\rho}:=\left\{ \theta \in \C^d / \Z^d \; | \; |\im \theta| < \rho \right\}, \quad \mathcal{B}_r:=\{ I\in \C^d \; | \; |I|< r \} \] 
with $\im \theta:=(\im \theta_1, \dots, \im \theta_d)$ and we let $B_r:=\{  I\in \R^d \; | \; |I|< r \}=\mathcal{B}_r \cap \R^d$. We denote by $\mathcal{A}_{\rho,r}$ the Banach space of holomorphic bounded functions $f : \T^{d}_{\rho} \times \mathcal{B}_{r}  \rightarrow \C$, which are real-valued for real arguments, equipped with the norm 
 \begin{equation*}
|f|_{\rho, r}:=\sup_{z\in \T^{d}_{\rho} \times \mathcal{B}_{r} }\abs{f(z)}.
 \end{equation*}
Here's our main result.

 \begin{Main}
\label{thm_main}
Assume $H \in \mathcal{A}_{\rho,r_0}$ is as in \eqref{eq_main_ham} with $\rho>0$, $r_0>0$, and $\omega\in DC(\tau, \gamma)$. Then there exist positive constants $r^*$, $c$ and $C$ which depend only on $d$, $\gamma$, $\tau$, $\rho$, $r_0$, $|H|_{\rho, r_0}$ and the BNF $N$ such that for all $0<r<r^*$, there is an open set $\mathcal{U}_r \subseteq \T^d  \times B_r$ with the measure estimate
 \begin{equation}
    \label{eq_meas_compl}
     \mathrm{Leb}(\T^d  \times B_r \setminus \mathcal{U}_r) < \mathrm{Leb}(\T^d  \times B_r)\exp(-c r^{-a})
   \end{equation} 
and such that for any $(\theta_0, I_0)\in \mathcal{U}_r$, there exists $\omega(I_0)\in \R^d$ and the corresponding solution $(\theta(t), I(t))$ satisfies
  \begin{equation}
   \label{eq_quasi_periodic}
    |\theta(t)-\theta_0-t\omega(I_0)| \leq Cr, \quad |I(t)-I_0|\leq Cr^2 \quad \text{for} \quad |t| \leq \exp\exp(cr^{-a}).
   \end{equation}
\end{Main}

Even though we did not state it, it will be clear from the proof that not only the vectors $\omega(I_0)$ are at a distance $Cr$ from $\omega\in DC(\tau, \gamma)$ but also they are at an exponentially small distance $\exp(-cr^{-a})$ from a subset of vectors in $DC(\bar \tau, \bar \gamma)$ for a proper choice of $\bar \tau>\tau$ and $\bar \gamma<\gamma$. To explain the dependence of the constants on the BNF and the strategy of the proof, recall that given $m \geq 2$, $N_m$ is the truncated BNF up to order $m$, and we extend this notation to $m=1$ by setting $N_1(I)=\omega\cdot I$. Now let $F_m=\nabla N_m$ be the gradient of $N_m$ for $m \geq 1$, which is a polynomial map of degree $m-1$, and define $V_m$ to be the vector space spanned by the partial derivatives of $F_m$ evaluated at the origin:
\begin{equation}\label{freq_space}
V_m := \mathrm{Vect}\{\partial_I^{\alpha} F_m(0), \; \alpha \in \N^d\}=\mathrm{Vect}\{\partial_I^{\alpha} F_m(0), \;  |\alpha| \leq m-1\}.
\end{equation}
These vector spaces form a non-decreasing sequence, $V_m \subseteq V_{m+1}$ for $m \geq 1$ and they are contained in $\R^d$: consequently this sequence is stationary and we can define $m^*\geq 1$ to be the smallest integer such that $V_m=V_{m^*}:=V$, for all $m\geq m^*$. Obviously, the polynomial map $F_{m^*} : \R^d \rightarrow V$ is non-degenerate in the sense that its partial derivatives generate $V$, and the dependence of the constants on the BNF in Theorem~\ref{thm_main} will only depend on $N_{m^*}$. To simplify the notations and statements in the sequel, we shall call ``constants" any positive constant which depends only on $d$, $\gamma$, $\tau$, $\rho$, $r_0$, $|H|_{\rho, r_0}$ and the truncated BNF $N_{m^*}$; in particular, a statement valid for $r>0$ small enough means that there exists a constant $r^*$ for which the statement is valid for all $0<r < r^*$. As we shall see, the main difficulty in proving Theorem \ref{thm_main} is when the dimension $l$ of $V$ satisfies $2 \leq l \leq d-1$ (which only occurs when $d \geq 3$) and the proof will consist in the following steps:
\begin{itemize}
 \item[1)] First, for $r>0$ small enough, we apply a BNF normalization up to optimal order $m=m(r)$
\[ H\circ \Psi_m(\theta, I) = N_{m} (I) + P_m(\theta,I) \]
where the size of the remainder $P_m(\theta,I)=\mathcal{O}(\abs{I}^{m+1})$ is bounded by $\mu=\exp(-c r^{-a})$. We shall deduce this from a general statement taken from~\cite{jorba}, and this is the content of Theorem~\ref{thm_jorba} in~\S\ref{sec_bir}. 
\item[2)] For any $r>0$ small enough, we have $m(r)\geq m^*$ and the polynomial mapping $F_{m(r)}=\nabla N_{m(r)} : \R^d \rightarrow V_{m(r)}=V$, which we denote by $F_r$ for simplicity, is non-degenerate and close to $F_{m^*}$; using this and the fact that $V \cap DC(\tau,\gamma)$ is non-empty (as it contains $\omega$) we will show that for some sufficiently large $\bar \tau>\tau$ and sufficiently small $\bar \gamma<\gamma$, the set $S_r$ of points $I_0 \in B_r$ for which $F_r(I_0) \in DC(\bar \tau,\bar \gamma)$ has positive measure (more precisely, we will show that the complement of this set has a measure bounded by a constant times a power of $\bar \gamma$). This will be stated in Proposition~\ref{arit} in~\S\ref{sec_arithmetic}.
 \item[3)] Consider the Hamiltonian $\bar H_r=H \circ \Psi_{m}$ with $m=m(r)$ given by the first step, restricted to the $\mu^{1/2}$-neighborhood of the set $\T^d \times S_r$ given by the second step with the value $\bar \gamma=\mu^{1/4}$. On this domain $\bar H_r$ can be considered as a $\mu$-perturbation of an integrable Hamiltonian with frequencies in $DC(\bar \tau,\bar \gamma)$ and by classical results on integrable normal form up to an exponentially small remainder (similar but slightly more general than Birkhoff normalization), there exists a transformation $\Phi$ such that $\bar H_r \circ \Phi$ is integrable up to a remainder which is exponentially small in $(1/\mu)^{\bar a}$ for some $\bar a>0$. This will be the content of Theorem~\ref{pos} in~\S\ref{sec_exp}, which follows from a statement taken from~\cite{Pos93}.  
 \item[4)] The last step is a mere conclusion: $\bar H_r \circ \Phi=H \circ \Psi_{m} \circ \Phi$ is integrable up to an exponentially small remainder in $(1/\mu)^{a'}$, and thus doubly exponentially small in $(1/r)^{a}$. This will give stability estimates first for $\bar H_r \circ \Phi$, and then for $H$, for a time-scale which is doubly exponentially large in $(1/r)^{a}$. By the second step, the measure estimate of the solutions not covered by these stability estimates is a power of $\bar \gamma=\mu^{1/4}$ and consequently it is exponentially small in $(1/r)^{a}$. The corresponding details will be given in~\S\ref{sec_conclusion}.  
\end{itemize} 

It remains to discuss the cases where the dimension $l$ of $V$ satisfies $l=1$ or $l=d$. In the first case, we do have $V=\R\omega$ and the second step we described above is valid on a whole neighborhood of the origin and not only on a proper subset, and so are the estimates~\eqref{eq_quasi_periodic}.

\begin{Main}
\label{cor_dim1}
 Assume $H \in \mathcal{A}_{\rho,r_0}$ is as in \eqref{eq_main_ham} with $\omega\in DC(\tau, \gamma)$, and $V=\R\omega$. Then there exist positive constants $c$ and $C$ such that for $r$ small enough, for any $(\theta_0, I_0)\in \T^d \times B_{r}$, there exists $\omega(I_0)\in \R^d$  and the corresponding solution $(\theta(t), I(t))$ satisfies
  \begin{equation*}
    |\theta(t)-\theta_0-t\omega(I_0)| \leq Cr, \quad |I(t)-I_0|\leq Cr^2 \quad \text{for} \quad |t| \leq \exp\exp(cr^{-a}).
   \end{equation*}
\end{Main}

This result is in fact not surprising because it follows from a result proved independently by R{\"u}ssmann (\cite{russmann}) and Bruno (\cite{Bru71}) that in this case, the transformation to the BNF actually converges and so $H$ must be analytically conjugated to its linear part $\omega\cdot I$ (and the conjugacy is exact-symplectic up to a time-change).

Now in the case $l=d$, we do have $V=\R^d$, the BNF is actually non-degenerate and using only the first step of the scheme of the proof described above together with an abstract result of  R{\"u}ssmann (\cite{russ}) we can be sure that the open set $\mathcal{U}_r$ in Theorem~\ref{thm_main} actually contains a set $\mathcal{K}_r$ of DQP tori.   

\begin{Main}
\label{cor_dimd}
 Assume $H \in \mathcal{A}_{\rho,r_0}$ is as in \eqref{eq_main_ham} with $\omega\in DC(\tau, \gamma)$, and $V=\R^d$. Then there exists a constant $c>0$ such that for $r$ small enough, there is a set $\mathcal{K}_r \subseteq \T^d \times B_{r}$ of DQP tori with the measure estimate
 \begin{equation}
    \label{eq_meas_compl2}
     \mathrm{Leb}(\T^d \times B_r \setminus \mathcal{K}_r) < \mathrm{Leb}(\T^d \times B_r)\exp(-c r^{-a}).
   \end{equation} 
\end{Main}

Again, this result is not new since the existence of a set of DQP tori with Lebesgue density one has been already proved in this context in~\cite{EFK}; yet the extra information contained in Theorem~\ref{cor_dimd} is the exponentially small measure~\eqref{eq_meas_compl2} which shows that the abundance of tori is the same for Kolmogorov non-degenerate or R{\"u}ssmann non-degenerate BNF.

To conclude, let us point out that we have decided to focus on the dynamics near a DQP torus, but we expect our main result in Theorem~\ref{thm_main} to be valid near a Diophantine elliptic fixed point (see~\cite{BFN2} for results on double exponential stability near elliptic fixed points) with few modifications; indeed, the result we use in the first step is also known in this case (and follows again from~\cite{jorba}) and after performing such a transformation, upon removing a subset with exponentially small measure in phase space, one can use ``non-singular" analytic action-angles coordinates and the rest of the proof follows exactly in the same way (alternatively, one could only use elliptic coordinates but statements such as Proposition~\ref{prop_russ} and Theorem~\ref{pos} below would have to be re-proven in this setting). However, our proof does not extend to non-analytic Gevrey classes and we do not know if our main result holds true in this context; this may be related to the fact that Herman's conjecture is known to be false in these classes (see~\cite{EFK}).

\section{BNF with exponentially small remainder}
\label{sec_bir}

We first recall the existence of a truncated BNF up to an exponentially small remainder. From now on, we denote by $\Pi_\theta$ (resp. $\Pi_I$) the projection onto angle components (resp. action components). 

\begin{theorem}
\label{thm_jorba}
Assume $H \in \mathcal{A}_{\rho,r_0}$ is as in \eqref{eq_main_ham} with $\omega\in DC(\tau, \gamma)$. Then there exist positive constants $c$ and $C$ such that for $r$ small enough, there exists an analytic exact-symplectic embedding $\Psi_r : \T^{d}_{\rho/2} \times \mathcal{B}_{3r}  \longrightarrow \T^{d}_{\rho} \times \mathcal{B}_{4r} $ with the estimates
\begin{equation}\label{estdist1}
|\Pi_\theta \Psi_r-\mathrm{Id}|_{\rho/2, 3r } \leq Cr, \quad |\Pi_I \Psi_r-\mathrm{Id}|_{\rho/2, 3r } \leq Cr^2
\end{equation}
such that
\[H \circ \Psi_{r}(\theta, I)= N_r(I)+ P_{r}(\theta, I)\]
with the estimate
\begin{equation}\label{est1}
|P_r|_{\rho/2, 3r} \leq \exp(-cr^{-a})
\end{equation}
and, for $m(r):=[c r^{-a}]$, we have $N_r=N_{m(r)}$ and the estimate
\begin{equation}\label{Birk1}
|N_r-N_m|_{3r} \leq Cr^{m+1}, \quad 1 \leq m \leq m^*+1.
\end{equation}
\end{theorem}

This result follows from the general statement of Theorem 3.8 in \cite{jorba}, up to slight changes in the notation (for instance, we use $r=R^2$ and we can discard elliptic variables) and in the numerical constants involved, and with the following modifications that we now describe, which follow directly from the proof of Theorem 3.8 in \cite{jorba}. 

First, in that reference, they have $N_r \neq N_{m(r)}$ but one can check that
\[ N_r(I)=N_{m(r)}(I)+\mathcal{O}(\abs{I}^{m(r)+1}),\]
that is the difference between $N_r$ and $N_m(r)$ is flat up to order $m(r)$; however, using this, an expansion of $N_r-N_{m(r)}$ into Taylor series (by analyticity) and Cauchy inequalities, one easily finds that $N_r-N_{m(r)}$ satisfy an estimate similar to the one for $P_r$ given in~\eqref{est1} (assuming $N_r$ is defined and bounded on a slightly larger domain, which can be assumed, and restricting $c$ if necessary) and thus upon replacing $P_r$ by $P_r+N_r-N_{m(r)}$, we can indeed assume $N_r = N_{m(r)}$ without affecting~\eqref{est1} and~\eqref{Birk1}. Then the estimates on the distance to the identity stated in~\eqref{estdist1} are slightly better than those stated in Theorem 3.8 (in that reference, $r$ is used to absorb the large positive constant $C$) but they clearly follow from the proof (indeed, $\Psi_r$ is obtained as a finite composition of transformations of the form $(\theta, I) \mapsto (\theta + \mathcal{O}(|I|), I + \mathcal{O}(|I|^2))$). Finally, the estimate~\eqref{Birk1} is stated only for $m=2$ in~\cite{jorba} but holds true for any given ``fixed" integer $m$ such that $1 \leq m \leq m^*+1$ (clearly here $r$ is small enough so that $2 \leq m^*+1 \leq m(r)$); indeed, and this is classical, such an estimate is obtained by applying $m$ steps of Birkhoff normalization with ``large" losses of widths of analyticity (depending on $r_0$ and $m^*$ but independent of $r$) and the remaining $m(r)-m$ steps with uniformly ``small" losses.

\section{Measure on the set of Diophantine points}
\label{sec_arithmetic}

In this section, we consider the integrable Hamiltonian $N_r=N_{m(r)}$, defined and analytic on the real domain $B_{3r}$ (so in particular $N_r$ is smooth on the closed ball $\bar B_{2r}$), which is the truncated BNF given by Theorem \ref{thm_jorba}, and we set $F_r:=\nabla N_r$. We shall restrict ourselves to the case where $2 \leq l \leq d-1$ where $l$ is the dimension of the space $V$; observe that necessarily $m^* \geq 2$ in this case. Using the fact that $F_r$ is analytic (indeed, it is a polynomial map) and $F_r(0)=\omega \in DC(\tau,\gamma)$, we shall prove that the set of points $I \in B_{2r}$, with $r$ small enough, for which $F_r(I)$ is Diophantine has a relatively large Lebesgue measure. Such results are well-known, see for instance \cite{kleinbock}, but we give a proof adapted to our context following~\cite{You}. 
 
\begin{proposition}
\label{arit}
Assume that $2 \leq l \leq d-1$. There exists a constant $C$ such that for $r$ and $\bar{\gamma}$ small enough, if we set $\bar \tau:=(m^*-1) (d+1)+\tau+1$, then we have
 \begin{equation}
 \label{eq_ordre_gamma}
\mathrm{Leb}(\{ I\in B_{2r} \; |\;  F_r(I)\notin DC(\bar \tau, \bar \gamma)\})\leq C \bar \gamma^{1/(m^*-1)}r^{d-1}.
 \end{equation}
\end{proposition}

We first recall that by definition of $V$, there exist $\alpha_1, \dots, \alpha_l \in \N^d$ such that the vectors $\partial_I^{\alpha_1} F_{m^*}(0), \dots, \partial_I^{\alpha_l} F_{m^*}(0)$ are linearly independent, and in view of~\eqref{freq_space}, we necessarily have $|\alpha_i| \leq m^* -1$. For $F_m=(F_{m,1}, \dots, F_{m,d})$, we set $F_m^l=(F_{m,1}, \dots, F_{m,l})$; without loss of generality we may assume that the square matrix of size $l$  
\begin{equation*}
A^{l}_{m^*}(I)= \left(\partial_I^{\alpha_1}F_{m^*}^l(I)^{\top} \; | \ldots | \;  \partial_I^{\alpha_l}F_{m^*}^l(I)^{\top} \right)
\end{equation*}
has a non-zero determinant at $I=0$, hence there exists a constant $\beta>0$ such that, denoting $S_l= \{\xi \in \R^l \; |\; |\xi|=1\}$, we have
\begin{equation}\label{min}
\min_{\xi \in S_l}|\xi \cdot A^{l}_{m^*}(0)| \geq 3\beta. 
\end{equation}
We have the following elementary lemma, where we consider the matrix
\begin{equation*}
A^{l}_{r}(I)= \left(\partial_I^{\alpha_1}F_{r}^l(I)^{\top} \; | \ldots | \;  \partial_I^{\alpha_l}F_{r}^l(I)^{\top} \right).
\end{equation*}

\begin{lemma}
\label{lemma_subespai}
For $r$ small enough, we have
\[ \min_{(I,\xi) \in \bar B_{2r} \times S_l}|\xi \cdot A^{l}_{r}(I)| \geq \beta \]
and as a consequence, for any $I \in \bar B_{2r}$,
\[ \mathrm{Vect}\{\partial_I^{\alpha_1} F^l_{r}(I), \dots, \partial_I^{\alpha_l} F^l_{r}(I)\}=\mathrm{Vect}\{\partial_I^{\alpha} F^l_{r}(I) \; | \; \alpha \in \N^d\}=V. \]
\end{lemma}

\begin{proof}
First, in view of~\eqref{min}, for $r$ small enough
\[ \min_{(I,\xi) \in \bar B_{2r} \times S_l}|\xi \cdot A_{m^*}^l(I)| \geq 2\beta. \]
Then observe that~\eqref{Birk1} with $m=m^*$ implies, for all $I \in \bar B_{2r}$ and $\alpha \in \N^d$ with $|\alpha| \leq m^*-1$, that
\[ |\partial_I^\alpha F_r(I)-\partial_I^\alpha F_{m^*}(I)| \leq Cr  \]
for some positive constant $C$ and consequently 
\[ \min_{(I,\xi) \in \bar B_{2r} \times S_l}|\xi \cdot A_{r}^l(I)| \geq \beta \]
for $r$ small enough. This proves the first part of the statement. For the second part of the statement, we have the inclusions:
\[ \mathrm{Vect}\{\partial_I^{\alpha_1} F_{r}(I), \dots, \partial_I^{\alpha_l} F_{r}(I)\} \subseteq \mathrm{Vect}\{\partial_I^{\alpha} F_{r}(I) \; | \; \alpha \in \N^d\} \subseteq V.\]
Indeed, the first one is obvious, while the second one follows from the fact that $F_r$ is analytic at $I=0$ together with the fact $V_{m(r)}=V$ since we are assuming $r$ small enough. We have just shown that the vectors $\partial_I^{\alpha_1} F_{r}(I), \dots, \partial_I^{\alpha_l} F_{r}(I)$ are linearly independent for $I \in \bar B_{2r}$ and hence they generate $V$ and the subspaces above are all equal.  
\end{proof}

Proposition~\ref{arit} will follow from the above lemma together with  the following Pyartli type inequality (the precise proposition below follows from Theorem 17.1 in \cite{russ}).

\begin{proposition}
\label{prop_russ}
Let $g : \bar B_{2r} \rightarrow \R$ be a function of class $C^{n+1}$ satisfying
\begin{equation*}
\min_{I\in \bar{B}_{2r}} \max_{0\leq j \leq n} \abs{D^{j}g(I)}\geq \beta 
\end{equation*}
 for some integer $n \geq 1$ and $\beta>0$. Then there exists $r'>0$ which depends only on $d$ such that for any $0<r<r'$ and any $0< \varepsilon \leq  \beta(2n+2)^{-1}$, we have
 \begin{equation}
  \label{eq_est_measure}
\mathrm{Leb}(\{I\in  B_{2r}\; | \; \abs{g(I)} \leq \varepsilon\}) \leq
    C |g|_{n+1}\varepsilon^{1/n}r^{d-1}\
 \end{equation}
where $|g|_{n+1}$ is the $C^{n+1}$ norm of $g$ and $C>0$ is a constant depending only on $d$, $n$ and $\beta$.
\end{proposition}

\begin{proofof}{\textbf{Proposition} \ref{arit}}
Recall that we are assuming $2 \leq l \leq d-1$. For $r$ small enough, Lemma~\ref{lemma_subespai} applies and since $\{\partial_I^{\alpha_1} F^l_{r}(I), \dots, \partial_I^{\alpha_l} F^l_{r}(I)\}$ are linearly independent and generate $V$ for all $I \in \bar B_{2r}$, there exist $b_{i,j}$ for $l+1 \leq i \leq d$ and $1 \leq j \leq l$ such that for all $I \in \bar B_{2r}$,
 \begin{equation}
  \label{eq_linear_comb}
  F_{r,l+1}(I)= \sum_{j=1}^{l}{b_{l+1, j} F_{r,j}(I)},\; \ldots \;, F_{r,d}(I)= \sum_{j=1}^{l}{b_{d,j} F_{r,j}(I)}. 
 \end{equation}
 We can also write $\omega=(\omega_1, \ldots, \omega_l, \sum_{j=1}^{l}{b_{l+1, j}\omega_j}, \ldots, \sum_{j=1}^{l}{b_{d,j}\omega_j})$, since $\omega \in V$, and using the fact that $\omega \in DC(\tau, \gamma)$ with $|\omega|=1$, we obtain, for all $k \in \Z^d \setminus\{0\}$,
 \begin{equation}
  \label{bound_omega}
   \sum_{j=1}^{l}\bigg{|}\sum_{i=l+1}^{d}{(b_{i,j}k_i+k_j)} \bigg{|} \geq \gamma |k|^{-\tau}.
 \end{equation}
Now for any $k\in \Z^d \setminus \{0\}$, consider the vector $\xi_k:=(\xi_{k,1}, \ldots, \xi_{k,l}) \in S_l$  defined by 
 \begin{equation*}
  \xi_{k,j}:= \frac{\sum_{i=l+1}^{d}{b_{i,j}k_i+k_j}}{\left(\sum_{j=1}^{l}{\abs{\sum_{i=l+1}^{d}{(b_{i,j}k_i+k_j)}}} \right)}, \quad 1 \leq j \leq l,
 \end{equation*}
and consider also the map $g_{r,k}: \bar B_{2r} \to \R$ defined by
 \begin{equation*}
 g_{r,k}(I)=\xi_k \cdot F_r^l(I)= \sum_{j=1}^l {\xi_{k,j} F_{r,j} (I)}.
 \end{equation*} 
It follows from~\eqref{eq_linear_comb} and~\eqref{bound_omega} that for all $I \in \bar B_{2r}$
 \begin{equation*}
|k \cdot F_r(I)|  \geq \gamma |k|^{-\tau} |g_{r,k}(I)|
\end{equation*}
 and thus, if we define, for all $k \in \Z^{d} \setminus \{0\}$ and some $\bar{\gamma}>0$ and $\bar \tau >0$,
\[ U_k:=\left\{I\in B_{2r} \; | \; |k \cdot F_r(I)| < \bar \gamma |k|^{-\bar \tau} \right\}\] 
and
\[V_k:=\left\{I\in B_{2r} \; |\; |g_{r,k}(I)|<\varepsilon_k:=\bar{\gamma}\gamma^{-1} |k|^{-\bar \tau+\tau}\right\} \]
we have the inclusion $U_k \subseteq V_k$ and consequently
\begin{align}
&\mathrm{Leb}(\{ I\in B_{2r} \; |\;  F_r(I)\notin DC(\overline{\tau}, \overline{\gamma})\})\leq \sum_{k \in \Z^{d} \setminus \{0\}} \mathrm{Leb}(U_k) \nonumber \\
& \leq \sum_{k \in \Z^{d} \setminus \{0\}} \mathrm{Leb}(V_k). \label{mes}  
\end{align}
It remains to estimate the Lebesgue measure of $V_k$. Clearly, the function $g_{r,k}$ is of class $C^{m^*}$ with a $C^{m^*}$-norm bounded due to the estimate in \eqref{Birk1} with $m=m^*$, and from Lemma~\ref{lemma_subespai} we have
\begin{align*}
&\min_{I\in \bar{B}_{2r}} \max_{0\leq j \leq m^*-1} \abs{D^{j}g_{r,k}(I)}\geq \min_{I \in \bar{B}_{2r}}\max_{j=1, \ldots, l}\abs{\partial_I^{\alpha_j}g_{r,k}(I)} \\
&\geq  \min_{(I,\xi) \in \bar B_{2r} \times S_l}|\xi \cdot A^{l}_{r}(I)| \geq \beta.
\end{align*}
For $r$ small enough, and choosing $\bar \gamma$ small enough so that $\varepsilon_k=\bar{\gamma}\gamma^{-1} |k|^{-\bar \tau+\tau} \leq \beta(2m^*)^{-1}$ for all $k \in \Z^d \setminus \{0\}$, Proposition \ref{prop_russ} applies to $g=g_{r,k}$ (and $n=m^*-1 \geq 1$) and yields the estimate   
\[ \mathrm{Leb}(V_k) \leq
    C\varepsilon_k^{1/(m^*-1)}r^{d-1}\leq C\bar{\gamma}^{1/(m^*-1)} |k|^{(-\bar \tau+\tau)/(m^*-1)}r^{d-1} \]
which, together with~\eqref{mes}, yields the wanted estimate because of our choice of $\bar \tau=(m^*-1) (d+1)+\tau+1$.
\end{proofof}

\section{Integrable normal forms near a set of Diophantine points}
\label{sec_exp}

In this section, we shall state the fact that a general $\mu$-perturbation of some integrable analytic Hamiltonian is, restricted to some neighborhood of Diophantine points, conjugated (symplectically and analytically) to some integrable normal form up to a remainder which is exponentially small with respect to $1/\mu$. Later on this will be applied to the Hamiltonian $H \circ \Psi_r$ given by Theorem~\ref{thm_jorba}, thus with $\mu$ exponentially small with respect to $1/r$, on the set of Diophantine points given by Proposition~\ref{arit} for some proper choice of $\bar \gamma$, eventually leading to a stability result which is doubly exponentially small with respect to $1/r$. 

To state this result precisely, we consider, for some positive parameters $\bar r$, $\bar \rho$, $M$ and $\mu$:
\begin{equation}\label{Ham2}
\begin{cases}
\overline{H}(\theta, I)=N(I) + P(\theta, I) \in \mathcal{A}_{\bar \rho,\bar r}, \\
|\nabla^2 N|_{\bar r} \leq M, \quad |P|_{\bar \rho,\bar r} \leq \mu
\end{cases}
\end{equation}
and given $0< \bar \gamma< \bar r$, $\bar \tau \geq d-1$, we consider the set
\begin{equation}\label{dioset}
 S=\{ I\in B_{\bar r-\bar \gamma} \; |\;  \nabla N(I)\in DC(\bar \tau, \bar \gamma)\}
\end{equation}
which is assumed to be non-empty: this is the set of $(\bar \tau,\bar \gamma)$-Diophantine points in $B_{\bar r}$ which are at a distance $\bar \gamma$ from the boundary of $B_{\bar r}$. In the terminology of~\cite{Pos93}, the set $S$ is completely $\alpha$, $K$-non resonant, for any $K \geq 1$ and $\alpha=\bar \gamma K^{-\bar \tau}$. Given $\delta>0$ we let
\[ V_\delta S:=\{I \in \C^d \; | \; d(I,S) < \delta\} \]
be the complex $\delta$-neighborhood of $S$: we shall have $\delta<\bar \gamma$ below so that $V_\delta S \cap \R^d$ will be indeed included in $B_{\bar r}$. By a slight abuse of notation we simply denote by $\norm{\cdot}_{\rho, \delta}$ the uniform norm for functions analytic and bounded on $\T^d_{\rho}\times V_\delta S$. The following statement follows from the normal form lemma of~\cite{Pos93}.

\begin{theorem}\label{pos}
Assume $\overline{H}$ is as in \eqref{Ham2} with $S\neq \emptyset$ as in~\eqref{dioset}. Then there exist positive constants $\bar c$ and $\bar C$ which depend only on $d$, $\overline{\tau}$, $\overline{\rho}$ and $M$ such that setting 
\[ \delta:=\mu^{1/2}, \quad \nu:=\left(\bar{c}^{-1} \bar{\gamma}^{-1}  \delta\right)^{\bar a}, \quad \bar a:=1/(\bar \tau+1) \]
and assuming that
\begin{equation}\label{seuil}
\nu<1
\end{equation}
there exists an exact-symplectic embedding 
\[ \Phi :  \T^d_{\bar \rho/6} \times V_{\delta/2} S \longrightarrow \T^d_{\bar \rho} \times V_{\delta}S\] 
with the estimates
\begin{equation}\label{estdist2}
\norm{\Pi_\theta\Phi-\mathrm{Id}}_{\bar \rho/6,\delta/2} \leq \nu, \quad \norm{\Pi_I\Phi-\mathrm{Id}}_{\bar \rho/6,\delta/2} \leq \nu \delta  
\end{equation}
such that
\[  \overline{H} \circ \Phi=N(I) +G(I)+ R(\theta, I)  \]
with the estimates
\begin{equation}\label{est2}
\norm{G}_{\delta/2} \leq \overline{C}\mu, \quad \norm{R}_{\bar \rho/6,\delta/2} \leq \overline{C}\mu\exp(-1/\nu).
\end{equation}
\end{theorem}

\section{Proofs of the main results}
\label{sec_conclusion}

We are now ready to conclude the proof of our main Theorem~\ref{thm_main}, as a consequence of Theorem~\ref{thm_jorba}, Proposition~\ref{arit} and Theorem~\ref{pos}.
\begin{proofof}{\textbf{Theorem} \ref{thm_main}}
It is sufficient to prove the statement when the dimension $l$ of the space $V$ satisfies $2 \leq l \leq d-1$; the cases $l=1$ and $l=d$ will be consequences of respectively Theorem~\ref{cor_dim1} and Theorem~\ref{cor_dimd} below. We shall assume $r$ small enough finitely many times in the sequel, without explicitly mentioning it, and we shall denote by $c$ and $C$ positive constants which may vary from line to line. First, Theorem~\ref{thm_jorba} applies and yields an exact-symplectic embedding  $\Psi_r : \T^{d}_{\rho/2} \times \mathcal{B}_{3r}  \longrightarrow \T^{d}_{\rho} \times \mathcal{B}_{4r}$ with the estimates~\eqref{estdist1} such that
\[ H \circ \Psi_{r}(\theta, I)= N_r(I)+ P_{r}(\theta, I)\] 
with the estimates~\eqref{est1} on $P_r$ and~\eqref{Birk1} on $N_r$. We set
\begin{align*}
 &\bar r:=2r, \quad \bar \rho:=\rho/2, \quad \mu:=\exp(-cr^{-a}), \\
 &\bar \gamma:=\mu^{1/4}, \quad \bar \tau:=(m^*-1) (d+1)+\tau+1. 
 \end{align*}
Observe that any fixed positive power of $\mu$, even multiplied by a large positive constant or a fixed negative power of $r$, is bounded by $\exp(-cr^{-a})$ for some appropriate constant $c$, provided $r$ is small enough. Then Proposition~\ref{arit} applies and gives the estimate
\[ \mathrm{Leb}(\{ I\in B_{\bar r}=B_{2r} \; |\;  F_r(I)\notin DC(\bar \tau, \bar \gamma)\})\leq C \bar \gamma^{1/(m^*-1)}r^{d-1} \]
and consequently, the set
\[ S_r=\{ I\in B_{\bar r- \bar \gamma} \; |\;  \nabla N_r(I)=F_r(I)\in DC(\bar \tau, \bar \gamma)\} \]
is non-empty and we have the measure estimate
\begin{equation}\label{measure}
\mathrm{Leb} (B_{2r} \setminus S_r) \leq  C \bar \gamma^{1/(m^*-1)}r^{d-1}+C\bar \gamma r^{d-1} \leq C\exp(-cr^{-a})r^{d-1}
\end{equation}
Next we want to apply Theorem~\ref{pos} to
\[ \bar{H}:=H \circ \Psi_r, \quad N:=N_r, \quad P:=P_r \]
and it follows from~\eqref{Birk1} with $m=2$ that $|\nabla^2N_r|_{\bar r}$ is indeed bounded by a constant, and so are $\bar c$ and $\bar C$, therefore recalling that
\[ \delta=\mu^{1/2}, \quad \nu=\left(\bar{c}^{-1} \bar{\gamma}^{-1}  \delta\right)^{\bar a}, \quad \bar a=1/(\bar \tau+1) \]
we do have
\[ \nu=\left(\bar{c}^{-1} \mu^{1/4}\right)^{\bar a} \leq \exp(-cr^{-a}). \]
We have the inclusion $\T^{d}_{\bar \rho} \times V_\delta S_r \subseteq \T^{d}_{\rho/2} \times \mathcal{B}_{3r} $ and the condition~\eqref{seuil} holds true hence Theorem~\ref{pos} applies and yields an exact-symplectic embedding
\[ \Phi=\Phi_r :  \T^d_{\bar \rho/6} \times V_{\delta/2} S_r \longrightarrow \T^d_{\bar \rho} \times V_{\delta}S_r\] 
with the estimates~\eqref{estdist2} such that 
\[  \tilde{H}_r:=\overline{H} \circ \Phi_r=H \circ \Psi_r \circ \Phi_r=N_r(I) +G_r(I)+ R_r(\theta, I)  \]
with the estimates
\begin{equation}\label{doubleexp}
\norm{G_r}_{\delta/2} \leq \overline{C}\mu, \quad \norm{R_r}_{\bar \rho/6,\delta/2} \leq \overline{C}\mu\exp(-1/\nu).
\end{equation}
From now on, we shall only consider real domains and the transformations restricted to these domains. We define
\[ \tilde{\mathcal{V}}_r:=\T^d \times (V_{\delta/4}S_r \cap \R^d), \quad \mathcal{V}_r:=\Psi_r(\Phi_r(\tilde{\mathcal{V}}_r))  \]
and we shall prove first that the stability estimates hold true for any solution $(\theta(t),I(t))$ of $H$ with initial condition $(\theta_0,I_0) \in \mathcal{V}_r$. To do so, we consider the corresponding solution $(\tilde{\theta}(t),\tilde{I}(t))$ of $\tilde{H}_r=H \circ \Psi_r \circ \Phi_r$ with initial condition $(\tilde{\theta}_0,\tilde{I}_0) \in \tilde{\mathcal{V}}_r$ and we define
\begin{equation*}
\omega(I_0):= \nabla N_r(\tilde{I}_0)+\nabla G_r(\tilde{I}_0) \in \R^d. 
\end{equation*} 
Since the transformations are symplectic, $\Psi_r(\Phi_r(\tilde{\theta}(t),\tilde{I}(t)))=(\theta(t),I(t))$ as long as the solutions are defined, and from~\eqref{estdist1} and~\eqref{estdist2} we have
\[ |\tilde{I}(t)-I(t)| \leq Cr^2+\delta\nu \leq Cr^2, \quad |\tilde{\theta}(t)-\theta(t)| \leq Cr+\nu \leq Cr.  \]
From the estimates~\eqref{Birk1} (again with $m=2$) and~\eqref{doubleexp}, together with Cauchy inequalities, we have
\begin{equation}\label{bout1}
|\nabla^2 N_r|_{2r} \leq C, \quad \norm{\nabla^2 G_r}_{\delta/3} \leq C
\end{equation}
and if we set
\[ T:=\exp(1/\nu) \geq \exp\left(\exp(cr^{-a})\right) \]
then
\begin{equation}\label{bout2}
\norm{\partial_IR_r}_{\bar \rho/6,\delta/3} \leq T^{-1}, \quad \norm{\partial_\theta R_r}_{\bar \rho/7,\delta/2} \leq T^{-1}.
\end{equation}
From the Hamiltonian equations associated to $\tilde{H}$ and the second inequality of~\eqref{bout2}, we have
\begin{equation}\label{bout3}
 |\tilde{I}(t)-\tilde{I}_0| \leq T^{-2/3} \quad \text{for} \quad |t| \leq T^{1/3}
\end{equation}
and therefore
\[ |I(t)-I_0| \leq T^{-2/3} +2Cr^2 \leq C r^2 \quad \text{for} \quad |t| \leq T^{1/3}.  \]
Then using~\eqref{bout1},~\eqref{bout3} and the definition $\omega(I_0)$, we have
\[ |\partial_IN_r(\tilde{I}(t))+\partial_IG_r(\tilde{I}(t))-\omega(I_0)| \leq C T^{-2/3} \quad \text{for} \quad |t| \leq T^{1/3}, \] 
and from the Hamiltonian equations and the first inequality of~\eqref{bout2}, we also have
\[ |\dot{\tilde{\theta}}(t)-\partial_IN_r(\tilde{I}(t))-\partial_IG_r(\tilde{I}(t))| \leq T^{-2/3} \quad \text{for} \quad |t| \leq T^{1/3}. \]
Thus
\[ |\dot{\tilde{\theta}}(t)-\omega(I_0)| \leq CT^{-2/3} \quad \text{for} \quad |t| \leq T^{1/3}. \]
This last inequality implies
\[ |\tilde{\theta}(t)-\tilde{\theta}_0-t\omega(I_0)| \leq CT^{-1/3} \quad \text{for} \quad |t| \leq T^{1/3}, \]
which gives
\[ |\theta(t)-\theta_0-t\omega(I_0)| \leq CT^{-1/3}+2Cr \leq C r \quad \text{for} \quad |t| \leq T^{1/3}. \]
Since $T^{1/3} \geq \exp\left(\exp(c r^{-a})\right)$, this concludes the proof of the stability estimates, and now it remains to prove that $\mathcal{V}_r$ contains an open set $\mathcal{U}_r$ with the wanted relative measure estimate. To do so, we first observe that $\Phi_r(\tilde{\mathcal{V}}_r)$ is contained in $\T^d \times (V_{\delta}S_r \cap \R^d)$, hence it is contained in $\T^d \times B_{2r}$ but also, in view of~\eqref{estdist2}, $\Phi_r(\tilde{\mathcal{V}}_r)$ contains (for instance) $\T^d \times (V_{\delta/5}S_r \cap \R^d)$ hence it contains $\T^d \times S_r$ and thus from~\eqref{measure} we get
\begin{equation*}
\mathrm{Leb} (\T^d \times B_{2r} \setminus \Phi_r(\tilde{\mathcal{V}}_r))) \leq\mathrm{Leb} (\T^d \times B_{2r} \setminus \T^d \times S_r) \leq C\exp(-cr^{-a})r^{d-1}.
\end{equation*}
Then, from~\eqref{estdist1}, $\Psi_r$ is a Lipeomorphism of $\T^d \times B_{2r}$ onto its image, which contains $\T^d \times B_{r}$, so that if we define
\[ \mathcal{U}_r:=\mathcal{V}_r \cap \left(\T^d \times B_{r}\right)= \Psi_r(\Phi_r(\tilde{\mathcal{V}}_r)) \cap \left(\T^d \times B_{r}\right), \]
we finally get
\[ \mathrm{Leb} (\T^d \times B_{r} \setminus \mathcal{U}_r)  \leq C\exp(-cr^{-a})r^{d-1} \leq \exp(-cr^{-a}) \mathrm{Leb} (\T^d \times B_{r})   \]
and this concludes the proof.
\end{proofof}

\begin{proofof}{\textbf{Theorem} \ref{cor_dim1}}
The proof follows the same lines as in the proof of Theorem~\ref{thm_main}, the only difference being that since we are assuming $V=\R\omega$, then $F_r(I)=\nabla N_r(I)=\lambda_r(I)\omega$ with $\lambda_r(0)=1$, and from~\eqref{Birk1} with $m=2$, if follows that the real-valued function $\lambda_r$ is close to $1$ for $r$ small enough. We can therefore choose $\bar{\tau}=\tau$ and for instance $\bar{\gamma}=\gamma/2$ so that Proposition~\ref{arit} becomes useless since in this case, one can choose $S_r=B_{2r}$ and consequently $\mathcal{U}_r=\T^d \times B_{r}$.
\end{proofof}

\begin{proofof}{\textbf{Theorem} \ref{cor_dimd}}
The proof follows directly from Theorem~\ref{thm_jorba} and a general KAM theorem proved by R{\"u}ssmann. Indeed, as we already pointed out, the assumption that $V=\R^d$ means that $N_{m^*}$, and consequently $N_{r}$ for $r$ small enough, is R{\"u}ssmann non-degenerate; the main result of~\cite{russ} applies to a perturbation of size $\mu=\exp(-cr^{-a})$, and the set not covered by DQP tori is estimated by a constant times $\mu^{1/(2(m^*-1))}$ which is still bounded by a quantity of the form $\exp(-cr^{-a})$.    
\end{proofof}

\noindent \textbf{Acknowledgements.} The second author would like to thank Bassam Fayad and Maria Saprykina for useful discussions along this work, and has been supported by the Swedish Research Council (VR 2015-04012).


\begin{thebibliography}{10}

\bibitem{arnold}
V.~I. Arnold.
\newblock Proof of a theorem of {A}. {N}. {K}olmogorov on the preservation of
  conditionally periodic motions under a small perturbation of the
  {H}amiltonian.
\newblock {\em Uspehi Mat. Nauk}, 18(5 (113)):13--40, 1963.

\bibitem{Bir66}
G.~Birkhoff.
\newblock {\em Dynamical systems}.
\newblock American Mathematical Society, Providence, R.I., 1966.

\bibitem{BFN}
A.~Bounemoura, B.~Fayad, and L.~Niederman.
\newblock Superexponential stability of quasi-periodic motion in {H}amiltonian
  systems.
\newblock {\em Comm. Math. Phys.}, 350(1):361--386, 2017.

\bibitem{BFN2}
A.~Bounemoura, B.~Fayad, and L.~Niederman.
\newblock Super-exponential stability for generic real-analytic elliptic
  equilibrium points.
\newblock {\em Adv. Math.}, 366:107088, 30, 2020.

\bibitem{Bru71}
A.~Bruno.
\newblock Analytical form of differential equations {I}.
\newblock {\em Trans. Moscow Math. Soc.}, 25:131--288, 1971.

\bibitem{You}
R.-M. Cao and J.-G. You.
\newblock Diophantine vectors in analytic submanifolds of euclidean spaces.
\newblock {\em Sci. China Ser A}, 50(9):1334--1338, 2007.

\bibitem{douady}
R.~Douady.
\newblock Stabilit\'{e} ou instabilit\'{e} des points fixes elliptiques.
\newblock {\em Ann. Sci. \'{E}cole Norm. Sup. (4)}, 21(1):1--46, 1988.

\bibitem{EFK}
L.~H. Eliasson, B.~Fayad, and R.~Krikorian.
\newblock Around the stability of {KAM} tori.
\newblock {\em Duke Math. J.}, 164(9):1733--1775, 2015.

\bibitem{gerard}
G.~Farré and B.~Fayad.
\newblock Instabilities for analytic quasi-periodic invariant tori.
\newblock 2019.
\newblock Preprint.

\bibitem{herman_icm}
M.~Herman.
\newblock Some open problems in dynamical systems.
\newblock In {\em Proceedings of the {I}nternational {C}ongress of
  {M}athematicians, {V}ol. {II} ({B}erlin, 1998)}, number Extra Vol. II, pages
  797--808, 1998.

\bibitem{jorba}
A.~Jorba and J.~Villanueva.
\newblock On the normal behaviour of partially elliptic lower-dimensional tori
  of {H}amiltonian systems.
\newblock {\em Nonlinearity}, 10(4):783--822, 1997.

\bibitem{kleinbock}
D.~Kleinbock.
\newblock An extension of quantitative nondivergence and applications to
  {D}iophantine exponents.
\newblock {\em Trans. Amer. Math. Soc.}, 360(12):6497--6523, 2008.

\bibitem{kolmo}
A.~N. Kolmogorov.
\newblock On conservation of conditionally periodic motions for a small change
  in {H}amilton's function.
\newblock {\em Dokl. Akad. Nauk SSSR (N.S.)}, 98:527--530, 1954.

\bibitem{Kri19}
R.~Krikorian.
\newblock On the divergence of {B}irkhoff {N}ormal {F}orms.
\newblock 2019.
\newblock Preprint.

\bibitem{MG}
A.~Morbidelli and A.~Giorgilli.
\newblock Superexponential stability of {KAM} tori.
\newblock {\em J. Statist. Phys.}, 78(5-6):1607--1617, 1995.

\bibitem{moser}
J.~Moser.
\newblock On invariant curves of area-preserving mappings of an annulus.
\newblock {\em Nachr. Akad. Wiss. G\"{o}ttingen Math.-Phys. Kl. II},
  1962:1--20, 1962.

\bibitem{Nek77}
N.~Nekhoroshev.
\newblock An exponential estimate of the time of stability of nearly integrable
  {H}amiltonian systems.
\newblock {\em Russian Math. Surveys}, 32(6):1--65, 1977.

\bibitem{Nek79}
N.~Nekhoroshev.
\newblock An exponential estimate of the time of stability of nearly integrable
  {H}amiltonian systems {II}.
\newblock {\em Trudy Sem. Petrovs}, 5:5--50, 1979.

\bibitem{Pop00}
G.~Popov.
\newblock Invariant tori, effective stability, and quasimodes with
  exponentially small error terms. {I}. {B}irkhoff normal forms.
\newblock {\em Ann. Henri Poincar\'{e}}, 1(2):223--248, 2000.

\bibitem{Pos93}
J.~Pöschel.
\newblock Nekhoroshev estimates for quasi-convex {H}amiltonian systems.
\newblock {\em Math. Z.}, 213:187--216, 1993.

\bibitem{russmann}
H.~R{\"u}ssmann.
\newblock \"{U}ber die {N}ormalform analytischer {H}amiltonscher
  {D}ifferentialgleichungen in der {N}\"{a}he einer {G}leichgewichtsl\"{o}sung.
\newblock {\em Math. Ann.}, 169:55--72, 1967.

\bibitem{russ}
H.~R{\"u}ssmann.
\newblock Invariant tori in non-degenerate nearly integrable {H}amiltonian
  systems.
\newblock {\em Regul. Chaotic Dyn.}, 6(2):119--204, 2001.

\bibitem{S54}
C.~L. Siegel.
\newblock \"{U}ber die {E}xistenz einer {N}ormalform analytischer
  {H}amiltonscher {D}ifferentialgleichungen in der {N}\"{a}he einer
  {G}leichgewichtsl\"{o}sung.
\newblock {\em Math. Ann.}, 128:144--170, 1954.
\end{thebibliography}

\end{document}